\def\lmjcorespemail{andrius.grigutis@mif.vu.lt}
\def\lmjrunauthor{Andrius Grigutis, Juozas Petkelis}
\def\lmjruntitle{The limit law of the maximum of discrete partial-sums distribution II}

\def\lmjmsc{60G50, 60G70, 60E10}
\def\lmjkeywords{maximum distribution, random walk, limit law, Vandermonde matrix, linear recurrence, initial values, biased Rademacher random walk}

\newcommand\dist{\buildrel d \over =}
\newcommand\al{\alpha}
\newtheorem{ex}{Example}

\documentclass{lmj}
\include{lmj.def}
\usepackage{xcolor}
\usepackage{lipsum}
\usepackage{comment}
\usepackage{bm}
\usepackage{bbm}
\usepackage{caption}
\usepackage{float}
\LMJCLS%

\begin{document}
\bibliographystyle{plainlmj}
\begin{topmatter}
 \title{The limit law of the maximum of discrete partial-sums distribution II}
 \author{Andrius Grigutis}
 \institution{Vilnius University, Faculty of Mathematics and Informatics, Institute of Mathematics}
 \email{andrius.grigutis@mif.vu.lt}
 \author{Juozas Petkelis}
 \institution{Vilnius University, Faculty of Mathematics and Informatics, Institute of Mathematics}
 \email{juozas.petkelis@mif.stud.vu.lt}%
%
\Received
 \end{topmatter}
\LMJarticle 
\begin{abstract}
Let $X_1,\,X_2,\,\ldots,\,X_N$, $N\in\mathbb N$ be independent, discrete, integer-valued random variables. Assume that $X_j\geqslant m_j$ almost surely for each $j=1,\,2,\,\ldots,\,N$, where $m_1,\,m_2,\,\ldots,\,m_N\in\mathbb{Z}$ satisfy $m_1+\cdots+m_N<0$. Furthermore, suppose that the sequence $X_1,\,X_2,\,\ldots$ is periodic in distribution, i.e. $X_k\dist X_{k+N}$ for all $k\in\mathbb N$. We derive computable representations for the distribution functions of $\max\{X_1,\,X_1+X_2,\,\ldots\}$, $\max\{X_2,\,X_2+X_3,\,\ldots\}$, $\ldots$, $\max\{X_N,\,X_N+X_{N+1},\,\ldots\}$.
The obtained formulas are based on a linear recurrence whose initial values are determined from a linear system that involves the roots of an associated characteristic equation and the distributions of $X_1,\,X_2,\,\ldots,\,X_N$. Several examples are presented, including a biseasonal-biased Rademacher random walk for which the distribution, generating functions, and all moments admit explicit closed-form expressions. In addition, we identify and correct several inaccuracies in the results reported in \cite{Grigutis2024}.
 \end{abstract}

\Keywords 
%
\section{Introduction}\label{s:1}

\subsection{From random walk to regularity}

Random walks and their extrema play a fundamental role in probability theory and its applications, including queuing theory, ruin theory, and reliability analysis. A classical problem is to determine the distribution of
$\sup_{n\geqslant 1}S_n$, where $S_n=X_1+\cdots+X_n$ is a random walk. If the random variables $X_1,\,X_2,\,\ldots$ are i.i.d. and have a negative mean, then $S_n\to -\infty,\,n\to\infty$ almost surely, and therefore $\sup_{n\geqslant 1}S_n<+\infty$ with probability one. Determining the distribution of $\sup_{n\geqslant 1}S_n$ has attracted much scientific attention and led to a variety of methods based on the Pollaczek--Khinchine formula, generating-function methods, Wiener--Hopf factorization, and related techniques.

One of the most remarkable phenomena in probability theory is that suitably combining an increasing number of random variables often leads to greater regularity rather than greater disorder. Classic examples include the law of large numbers and the central limit theorem, which describe the asymptotic behavior of sums of independent random variables under appropriate assumptions. Similar phenomena arise in many other probabilistic settings, where increasing combinations of random variables admit tractable limiting distributions. The distribution of $\sup_{n\geqslant 1}S_n$ provides another example of such behavior: although it depends on infinitely many random variables, it often possesses a surprisingly explicit structure.

\subsection{Introduction to the article, notations}\label{sub:notations}

Let $X_1,\,X_2,\,\ldots,\,X_N$, $N\in\mathbb N$ be independent, discrete, integer-valued random variables whose supports start at $m_1,\,m_2,\,\ldots,\,m_N\in\mathbb{Z}$, respectively, and $m_1+m_2+\cdots+m_N<0$. For each $j=1,\,2,\,\ldots,\,N$ define

\begin{align}\label{dist_f}
F^{(j)}_{\infty}(x):=\lim_{T\to\infty}\mathbb{P}\left(\bigcap_{n=j}^{T}\left\{\sum_{k=j}^{n}X_k\leqslant x\right\}\right)=
\mathbb{P}\left(\sup_{ n \geqslant j}\sum_{k=j}^{n}X_k\leqslant x \right),\,x\in\mathbb{Z},
\end{align}
where the sequence $X_1,\,X_2\,\ldots$ is assumed to be $N$-periodic in distribution, i.e. $X_k\dist X_{k+N}$ for all $k\in\mathbb{N}$. In this work, we derive explicit formulas for the distribution functions \eqref{dist_f} via a computable linear recurrence, whose coefficients are determined by the distributions $X_1,\,X_2,\,\ldots,\,X_N$; see Theorem \ref{T2}. To state the result precisely, we first introduce some notations. For the introduced integers $m_1,\,m_2,\,\ldots,\,m_N$, let:
\begin{align*}
D&:=m_1+m_2+\ldots+m_N<0,\\
M_j&:=\max\{m_j,\,m_j+m_{j+1},\,\ldots,\,m_j+m_{j+1}+\ldots+m_{j+N-1}\},\quad j=1,\,2,\,\ldots,\,N,
\end{align*}
where $m_{i+N}=m_{i}$ for all $i\in\mathbb{N}$.
From the definition \eqref{dist_f} it is clear that
\begin{align*}
F^{(j)}_{\infty}(x)=0,\quad x<M_j,\quad j=1,\,2,\,\ldots,\,N.
\end{align*}
Let
\begin{align*}
p^{(j)}_k:=\mathbb{P}(X_j=k),\quad
F_{X_j}(k):=\mathbb{P}(X_j\leqslant k),\quad k\in\mathbb{Z},\quad j=1,\,2,\,\ldots,\,N,
\end{align*}
and
\begin{align*}
S_N:=X_1+X_2+\ldots+X_N,\quad
f_N(k):=\mathbb{P}(S_N=k),\quad
F_{S_N}(k):=\mathbb{P}(S_N\leqslant k),\quad k\in\mathbb{Z}.
\end{align*}

Since the supports of $X_j$ start at $m_j$, we have $p^{(j)}_k=0$ for $k<m_j$ and $f_N(k)=0$ for $k<D$.

We also denote 
\begin{align*}
p^{(i_1,\,i_2,\,\ldots,\,i_r)}_k:=\mathbb{P}(X_{i_1}+X_{i_2}+\ldots+X_{i_r}=k),\quad k\in\mathbb{Z}
\end{align*}
where $r\in\{1,\,2,\,\ldots,\,N\}$ and $1\leqslant i_1<i_2<\ldots<i_r\leqslant N$. In particular,
\begin{align*}
p^{(1,\,2)}_k=\mathbb{P}(X_1+X_2=k),\quad p^{(1,\,3)}_k=\mathbb{P}(X_1+X_3=k),\quad p^{(1,\,2,\,N)}_k=\mathbb{P}(X_1+X_2+X_N=k),
\end{align*}
etc. Thus, the indices appearing in the superscript indicate which random variables out of $X_1,\,X_2,\,\ldots,\,X_N$ are included in the sum, and $r$ denotes the number of indices appearing in the superscript.

For the introduced random variables $X_1,\,X_2,\,\ldots,\,X_N$ we define
\begin{align*}
\mathcal{M}_j&:=\max\{0,\,X_j,\,X_j+X_{j+1},\,\ldots\},\quad j=1,\,2,\,\ldots,\,N.
\end{align*}
It is clear that $\mathcal{M}_1,\,\mathcal{M}_2,\,\ldots,\,\mathcal{M}_N$ are non-negative and integer valued random variables. If we denote their probability mass functions as
\begin{align}\label{pmf}
\pi_k^{(j)}:=\mathbb{P}(\mathcal{M}_j=k),\quad k\in\mathbb{N}_0,\quad j=1,\,2,\,\ldots,\,N,
\end{align}
then
\begin{align*}
\pi_0^{(j)}=F_{\infty}^{(j)}(0),\quad
\pi_k^{(j)}=F_{\infty}^{(j)}(k)-F_{\infty}^{(j)}(k-1),\quad k\in\mathbb{N},
\end{align*}
for all $j=1,\,2,\,\ldots,\, N$. Also

\begin{align*}
F^{(j)}_{\infty}(x)=\mathbb{P}(\mathcal{M}_j\leqslant x),\quad x\in\mathbb{N}_0, \quad j=1,\,2,\,\ldots,\,N.
\end{align*}

Besides the mentioned Theorem \ref{T2}, we prove identities of the probability generating functions of the random variables $\mathcal{M}_j$ and $X_j$, $j=1,\,2,\,\ldots,\,N$; see Proposition \ref{T1}. For an integer-valued random variable $X$ such that $\mathbb{P}(m\leqslant X<+\infty)=1$ we define its probability generating function by
\begin{align}\label{pgf}
G_X(s):=\sum_{j=m}^{\infty}s^j\mathbb{P}(X=j),
\end{align}
where $|s|\leqslant1,\,s\in\mathbb{C}$, and $s\neq0$ if $m<0$. It is easy to check that
\begin{align*}
\frac{G_X(s)}{1-s}=\sum_{n=m}^{\infty}\mathbb{P}(X\leqslant n)s^n,\quad |s|<1,
\end{align*}
and $s\neq0$ if $m<0$. Proposition 1 provides more information about the random variables $\mathcal{M}_1,\,\mathcal{M}_2,\,\ldots,\,\mathcal{M}_N$ than merely their distribution functions. 

Last but not least, we emphasize that Theorem \ref{T2} provides an explicit recurrent computational pattern of $F^{(j)}_{\infty}(x)$, $j=1,\,2,\,\ldots,\,N$ for $x\geqslant0$ only. Therefore, in Proposition \ref{T3} we state how to find the non-zero values of $F^{(j)}_{\infty}(x)$, $j=1,\,2,\,\ldots,\,N$ if $x<0$. Together with Lemma \ref{L1}, which identifies the required initial values, these results provide a complete recursive procedure for computing the distribution functions $F^{(j)}_{\infty}(x)$ and the associated characteristics of the maximum distributions.

\subsection{Prior work}\label{subsec:prior}

A similar distribution function $F^{(1)}_{\infty}$ has been investigated in \cite{gerve_grigutis_2024}. However, the periodic discrete random variables $X_1,\,X_2,\,\ldots,\,X_N$ were assumed to have identical supports, the algorithm to compute $F^{(1)}_{\infty}$ was less explicit than Theorem \ref{T2} in this work, and the context of occurrence of $F^{(1)}_{\infty}$ was largely motivated by applications in insurance mathematics.

In work \cite{Grigutis2024}, an attempt was made to compute $F^{(1)}_{\infty}$ explicitly. Upon a closer examination, several inaccuracies were identified in that paper. In particular, the statement on \cite[p.~482]{Grigutis2024} containing the phrase “which yields” should be interpreted as “which, if $\max\{S_1,\,S_2,\,\ldots,\,S_N\}=S_N$ almost surely, yields”. Let us explain why. The definition of $F^{(1)}_{\infty}$ and the law of total probability, see \cite[p. 482]{Grigutis2024}, imply 
\begin{align}\label{eq:recurrence}
F^{(1)}_{\infty}(x)=\hspace{-1cm}\sum_{\substack{m_1\leqslant i_1\leqslant x\\m_2\leqslant i_2\leqslant x-i_1\\
\vdots \vspace{1mm} \\
m_N\leqslant i_N\leqslant x-i_1-i_2-\ldots-i_{N-1}}}\hspace{-5mm}\mathbb{P}(X_1=i_1)\mathbb{P}(X_2=i_2)\cdots\mathbb{P}(X_N=i_N)\,
F^{(1)}_{\infty}\left(x-\sum_{j=1}^Ni_j\right),\quad x\in\mathbb{Z}.
\end{align}
Identity \eqref{eq:recurrence} is uncomfortable for the search of $F^{(1)}_{\infty}$. We may rewrite it as follows 
\begin{align*}
F_{\infty}^{(1)}(x)=\mathbb{P}(\max\{S_1,\,S_2,\,\ldots\}\leqslant x)
=\mathbb{P}(\max\{S_1,\,S_2,\,\ldots,\,S_N\}\leqslant x,\,\max\{S_{N+1},\,S_{N+2},\,\ldots\}\leqslant x).
\end{align*}
Due to $S_{N+k}-S_N\dist S_k$ for all $k\in\mathbb{N}$ and denoting $\tilde{S}_1,\,\tilde{S}_2,\,\ldots$ as independent copies of $S_1,\,S_2,\,\ldots$ correspondingly, we obtain
\begin{align*}
F_{\infty}^{(1)}(x)&=\mathbb{P}(\max\{S_1,\,S_2,\,\ldots,\,S_N\}\leqslant x,\,S_N+\max\{\tilde{S}_{1},\,\tilde{S}_{2},\,\ldots\}\leqslant x)\\
&=\sum_{k=D}^{x}\mathbb{P}(\max\{S_1,\,S_2,\,\ldots,\,S_N\}\leqslant x,\,S_N=k,\,\max\{\tilde{S}_{1},\,\tilde{S}_{2},\,\ldots\}\leqslant x-k)\\
&=\sum_{k=D}^{x}\mathbb{P}(\max\{S_1,\,S_2,\,\ldots,\,S_N\}\leqslant x,\,S_N=k)F^{(1)}_{\infty}(x-k).
\end{align*}

Since 
\begin{align*}
\mathbb{P}(S_N=k)=\mathbb{P}(S_N=k,\,\max\{S_1,\,S_2,\,\ldots,\,S_N\}\leqslant x)
+\mathbb{P}(S_N=k,\,\max\{S_1,\,S_2,\,\ldots,\,S_N\}>x)
\end{align*}
we get
\begin{align}\label{eq:recurrence_with_B}
F_{\infty}^{(1)}(x)=\sum_{k=D}^{x}f_N(k)F^{(1)}_{\infty}(x-k)
-\sum_{k=D}^{x}b_k(x)F^{(1)}_{\infty}(x-k)
,\,x\in\mathbb{Z},
\end{align}
where 
\begin{align*}
b_k(x)=\mathbb{P}(\max\{S_1,\,S_2,\,\ldots,\,S_N\}>x,\,S_N=k),\,x,\,k\in\mathbb{Z}.
\end{align*}

Notice that $\mathbb{P}(\max\{S_1,\,S_2,\,\ldots,\,S_N\}=S_N)=1$ implies $b_k(x)=0$ for all $k\leqslant x$ and $x\in\mathbb{Z}$.  Indeed,
\begin{align*}
b_k(x)=\mathbb{P}(\max\{S_1,\,S_2,\,\ldots,\,S_N\}>x,\,S_N=k)\leqslant\mathbb{P}(S_N>x,\,S_N=k)=0,
\end{align*}
for all $k\leqslant x$ and $x\in\mathbb{Z}$.

Unfortunately, the boundary terms $b_k(x)$ were erroneously omitted in \cite{Grigutis2024} for all distributions. However, the paper \cite{Grigutis2024} is credible for $N=1$ or such distributions that imply $b_k(x)=0$ in \eqref{eq:recurrence_with_B}. Except \cite[Example 2]{Grigutis2024}, which we revisit in detail in this work; see Section \ref{sec:Examples}. 

In general, the distribution of the maximum of a random walk remains a central theme in probability theory. Classical contributions include the works of Felix Pollaczek and Aleksandr Khinchin on waiting-time and supremum distributions \cite{Pollaczek1930}, \cite{Khintchine1932}, as well as Dennis Lindley's seminal work on the waiting-time recursion in single-server queues \cite{Lindley1952}. Later, Frank Spitzer developed the fluctuation theory of random walks \cite{Spitzer1976}, while William Feller provided an extensive treatment of random walks and their extrema \cite{Feller1971}. These works laid the foundations for a rich theory connecting random-walk maxima with generating functions, renewal theory, ladder variables, Wiener--Hopf factorization, queueing theory, ruin theory, etc.

Recent developments continue to study the distribution of random-walk maxima and related quantities by using generating-function and recursive techniques; see, for example, \cite{Asmussen2003}, \cite{AsmussenAlbrecher2010}, \cite{Vlasiou2007}, \cite{BoxmaKellaMandjes2023}. In ruin theory, many formulas for discrete-time risk models are recursive; see the survey paper \cite{LiLuGarrido2009} and the references therein. Random-walk techniques also play an important role in probabilistic models arising in analytic number theory, including the study of the extrema of the Riemann zeta function and related logarithmically correlated fields \cite{ArguinBeliusHarper2017}.

The present work contributes to this line of research by studying random walks with periodical distributions. In contrast to the classical setting of independent and identically distributed random variables, we derive computable representations for the distribution functions $F_\infty^{(1)},\,F_\infty^{(2)},\,\ldots,\,F_\infty^{(N)}$ under the periodicity assumption and show that these distribution functions admit representation through a linear recurrence determined by the characteristics of incorporated distributions. In several examples, this recurrence can be solved explicitly. In particular, Section 4 contains a bi-seasonal biased Rademacher model for which the limiting distribution, its probability masses, generating function, and moments admit explicit closed-form expressions.

\section{Main Results}\label{sec:results}
In this section, we present the main results of the paper. Their proofs are given in Section \ref{sec:proofs}. Throughout, we use the notations introduced in Subsection~\ref{sub:notations} and additionally denote the column vector
\begin{align*}
{\pmb G}_{\mathcal{M}}(s):=\left(G_{\mathcal{M}_1}(s),\,G_{\mathcal{M}_2}(s),\,\ldots,\,G_{\mathcal{M}_N}(s)\right)^T,
\end{align*}
the matrix
\begin{align*}
&{\pmb M}(s):=\\
&\begin{pmatrix}\nonumber
1&G_{X_1}(s)&G_{X_1+X_2}(s)&\ldots&G_{X_1+X_2+\ldots+X_{N-1}}(s)\\
G_{X_2+X_3+\ldots+X_N}(s)&1&G_{X_2}(s)&\ldots&
G_{X_2+X_3+\ldots+X_{N-1}}(s)\\
G_{X_3+X_4+\ldots+X_N}(s)&G_{X_1+X_3+X_4+\ldots+X_N}(s)&1&\ldots&G_{X_3+X_4+\ldots+X_{N-1}}(s)\\
\vdots&\vdots&\vdots&\ddots&\vdots\\
G_{X_{N-1}+X_N}(s)&G_{X_1+X_{N-1}+X_N}(s)&G_{X_1+X_2+X_{N-1}+X_N}(s)&\ldots&G_{X_{N-1}}(s)\\
G_{X_N}(s)&G_{X_1+X_N}(s)&G_{X_1+X_2+X_N}(s)&\ldots&1
\end{pmatrix},
\end{align*}
and the column vector
\begin{align*}
{\pmb B}(s):=\left(B_1(s),\,B_2(s),\,\ldots,\,B_N(s)\right)^T,
\end{align*}
where
\begin{align*}
B_j(s):=\sum_{k=m_j}^{-1}s^k\sum_{i=0}^{k-m_j}\pi^{(j+1)}_iF_{X_j}(k-i),\quad j=1,\,2,\,\ldots,\,N,
\end{align*}
and $\pi^{(N+1)}=\pi^{(1)}$. The following proposition establishes a system of generating-function identities satisfied by the random variables $\mathcal{M}_1,\,\mathcal{M}_2,\,\ldots,\,\mathcal{M}_N$.

\begin{proposition}\label{T1}
Assume that $\mathbb{E}S_N<0$. Then, for every $s\in\mathbb{C}$ such that $|s|<1$ and $G_{S_N}(s)
\neq1$, the generating functions satisfy
\begin{align}\label{eq:T1}
\frac{{\pmb G}_{\mathcal{M}}(s)}{1-s}=\frac{{\pmb M}(s){\pmb B}(s)}{G_{S_N}(s)-1}.
\end{align}
\end{proposition}

By comparing the coefficients of equal powers of $s$ in \eqref{eq:T1}, we obtain the main theorem of the paper, which provides an explicit computational pattern for the values $F_{\infty}^{(j)}$. Since the random variables $X_1,\,X_2,\,\ldots$ are integer-valued, in the sequel, we write $F_{\infty}^{(j)}(n)$ instead of $F_{\infty}^{(j)}(x)$ whenever the argument is an integer.

\begin{thm}\label{T2}
If $\mathbb{E}S_N<0$, then the distribution functions $F_{\infty}^{(1)},\,F_{\infty}^{(2)},\,\ldots,\,F_{\infty}^{(N)}$ satisfy
\begin{align*}
F_\infty^{(j)}(n)
=
\sum_{r=1}^{N}
\sum_{\ell=m_r}^{-1}
a_{n-\ell}^{I_{jr}}
\sum_{i=0}^{\ell-m_r}
\pi_i^{(r+1)}
F_{X_r}(\ell-i),
\qquad
j=1,\,2,\,\ldots,\,N,\quad n\in\mathbb{N}_0,
\end{align*}
where
\begin{align*}
I_{jr}:=
\begin{cases}
(j,j+1,\ldots,r-1), & j<r,\\
\varnothing, & j=r,\\
(1,\,\ldots,\,r-1,\,j,\,j+1,\,\ldots,\,N), & j>r,
\end{cases}
\end{align*}

\begin{align*}
a_n^{I_{jr}}=
\begin{cases}
\sum\limits_{k=\mu_{jr}}^{n+D}
p_k^{I_{jr}}
a_{n-k}, & n\geqslant \mu_{jr}-D, \\
0, & n < \mu_{jr}-D,
\end{cases}
\qquad
\mu_{jr}=\sum_{t\in I_{jr}}m_t,
\end{align*}
and 
\begin{align*}
p_k^{I_{jr}}
=
\mathbb{P}\!\left(
\sum_{t\in I_{jr}} X_t = k
\right),
\qquad
k\in\mathbb{Z},
\end{align*}
with the convention that, if $I_{jr}=\varnothing$, then
\begin{align*}
a_n^{\varnothing}=a_n,\qquad
\mu_{jr}=0,
\qquad
p_k^{\varnothing}
=
\begin{cases}
1, & k=0,\\
0, & k\neq 0,
\end{cases}
\end{align*}
where
\begin{align*}
a_{-D}=\frac{1}{f_N(D)}, \quad a_{n}=\frac{1}{f_N(D)}\left({\pmb 1}_{\{n\geqslant-2D\}}a_{n+D}-\sum_{k=1}^{n+D}f_N(k+D)a_{n-k}\right),\quad n \geqslant -D+1,
\end{align*}
and $\pi^{(N+1)}=\pi^{(1)}$.
\end{thm}

Theorem \ref{T1} implies the following corollary. 
\begin{cor}\label{C1}
Say that the assumptions of Theorem \ref{T1} hold and $N=1$.
Then
\begin{align*}
F_\infty^{(1)}(n)
=
\sum_{\ell=m_1}^{-1}
a_{n-\ell}
\sum_{i=0}^{\ell-m_1}
\pi_i^{(1)}
F_{X_1}(\ell-i),
\qquad n\in\mathbb N_0,
\end{align*}
where
\begin{align*}
a_{-m_1}=\frac{1}{p^{(1)}_{m_1}},\quad
a_n
=
\frac{1}{p^{(1)}_{m_1}}
\left(
{\pmb 1}_{\{n\geqslant-2m_1\}}a_{n+m_1}
-
\sum_{k=1}^{n+m_1}
p^{(1)}_{k+m_1}\,
a_{n-k}
\right),
\qquad n\geqslant -m_1+1.
\end{align*}
\end{cor}

The next statement provides an algorithm to compute $F^{(1)}_{\infty}(n),\,F^{(2)}_{\infty}(n),\,\ldots,\,F^{(N)}_{\infty}(n)$ when $n$ is not necessary non-negative.

\begin{proposition}\label{T3}
The distribution functions $F^{(1)}_{\infty},\,F^{(2)}_{\infty},\,\ldots,\,F^{(N)}_{\infty}$ satisfy
\begin{align}\label{eq1_T3}
F_{\infty}^{(j)}(n)
=
\sum_{k=m_j}^{n}
p_k^{(j)}
F_{\infty}^{(j+1)}(n-k),\quad j=1,\,2,\,\ldots,\,N, \quad n\geqslant M_j,
\end{align}
where $M_j=\max\{m_j,\,m_j+m_{j+1},\,\ldots,\,m_j+m_{j+1}+\ldots+m_{j+N-1}\}$, $m_{i+N}=m_{i}$ for all $i\in\mathbb{N}$, and $F_{\infty}^{(N+1)}=F_{\infty}^{(1)}$.

Also
\begin{align}\label{eq2_T3}
F_\infty^{(j)}(n+m_j)
=
\sum_{k=0}^{n}
\pi_k^{(j+1)}
F_{X_j}(m_j+n-k),
\quad
j=1,\,2,\,\ldots,\,N,
\qquad
n\geqslant0,
\end{align}
where $\pi^{(N+1)}=\pi^{(1)}$.
\end{proposition}

Proposition \ref{T1} and Theorem \ref{T2} depend on the initial probabilities
\begin{align}\label{unknowns}
\pi^{(j)}_0,\,\pi^{(j)}_1,\,\ldots,\,\pi^{(j)}_{-m_j-1}, \quad j=1,\,2,\,\ldots,\,N. 
\end{align}
We next discuss how to find these values. If $b_k(x)=0$ for all $x$ and $k$ in \eqref{eq:recurrence_with_B}, then these initial values can be explicitly computed by eqs. (1.10)--(1.13) from \cite{Grigutis2024}. If $b_k(x)\neq0$ at least for some $x$ and $k$ in \eqref{eq:recurrence_with_B}, then the next lemma is a key statement for finding the required initial probabilities.

\begin{lem}\label{L1}
If $\mathbb{E}S_N<0$ and $s\in\mathbb{C}$ is such that $|s|\leqslant 1$ and $s\neq0$, then

\begin{align}\label{eq:memory_compact}\nonumber
&(1-s)\sum_{j=1}^{N}
s^{m_j}
G_{X_N+X_1+\cdots+X_{j-1}}(s)
\sum_{i=0}^{-m_j-1}\pi_i^{(j+1)}
\sum_{r=0}^{-m_j-1-i}
s^{r+i}F_{X_j}(m_j+r)\\
&=
G_{\mathcal M_N}(s)\left(G_{S_N}(s)-1\right),
\end{align}
where
\begin{align*}
G_{X_N+X_1+\cdots+X_{0}}(s):=1 \quad \text{and} \quad \pi^{(N+1)}=\pi^{(1)}.
\end{align*}
Moreover,
\begin{align}\label{eq:memory_compact_E}
\sum_{j=1}^{N}
\sum_{i=0}^{-m_j-1}
\pi_i^{(j+1)}
\sum_{r=0}^{-m_j-1-i}
F_{X_j}(m_j+r)
=
-\mathbb{E}S_N.
\end{align}
\end{lem}

The use of Lemma \ref{L1} is as follows. As $G_{\mathcal M_N}(s)$ is the probability generating function of the finite random variable $\mathcal{M}_N$, the right hand side of \eqref{eq:memory_compact} vanishes if $s$ is such that $G_{S_N}(s)=1$. There are exactly $-D-1$, $-D=-m_1-m_2-\ldots-m_N>0$ roots of $G_{S_N}(s)=1$ counted with their multiplicities in $|s|\leqslant1$, $s\notin\{0,\,1\}$; see Lemma 4 in \cite{gerve_grigutis_2024}. Evaluating \eqref{eq:memory_compact} at the zeros $\alpha_1,\,\alpha_2\,\ldots,\,\alpha_{-D-1}$ of $G_{S_N}(s)-1$, appending \eqref{eq:memory_compact_E}, and collecting coefficients of the unknown initial probabilities \eqref{unknowns}, yields the following $(-D)\times (-D)$ system of linear equations
\begin{align}\label{main_system}
{\pmb A}
{\pmb \pi}
=
{\pmb E},
\end{align}
where 
\begin{align*}
{\pmb \pi}
&=\left(\pi^{(1)}_0,\,
\pi^{(1)}_1,\,
\ldots,\,
\pi^{(1)}_{-m_N-1},\,
\pi^{(2)}_0,\,
\pi^{(2)}_1,\,
\ldots,\,
\pi^{(2)}_{-m_1-1},\,
\ldots,\,
\pi^{(N)}_0,\,
\pi^{(N)}_1,\,
\ldots,\,
\pi^{(N)}_{-m_{N-1}-1}\right)^T,\\
{\pmb E}
&=\left(0,\,0,\,\ldots,\,0,\, -\mathbb{E}S_N\right)^T_{(-D)\times 1},
\end{align*}
and the system's matrix ${\pmb A}$ is built as follows: 
\begin{align*}
{\pmb A}=\left(
A_1\circ B_1,\,A_2\circ B_2,\,\ldots,\,A_N\circ B_{N}\right),
\end{align*}

where $\circ$ denotes the Hadamard matrix product and, with $C:=-D>0$,
\begin{align*}
&\left(A_1\right)_{(-D)\times (-m_N)}:=\\
&\hspace{-1cm}\begin{pmatrix}
\sum\limits_{j=0}^{-m_N-1}\al_1^jF_{X_N}(m_N+j)
&\sum\limits_{j=0}^{-m_N-2}\al_1^{j+1}F_{X_N}(m_N+j)
&\ldots
&\sum\limits_{j=0}^{1}\al_1^{j-m_N-2}F_{X_N}(m_N+j)
&\al_1^{-m_N-1}p^{(N)}_{m_N}\\
\vdots&\vdots&\ddots&\vdots&\vdots\\
\sum\limits_{j=0}^{-m_N-1}\al_{C-1}^jF_{X_N}(m_N+j)
&\sum\limits_{j=0}^{-m_N-2}\al_{C-1}^{j+1}F_{X_N}(m_N+j)
&\ldots
&\sum\limits_{j=0}^{1}\al_{C-1}^{j-m_N-2}F_{X_N}(m_N+j)
&\al_{C-1}^{-m_N-1}p^{(N)}_{m_N}\\
\sum\limits_{j=0}^{-m_N-1}F_{X_N}(m_N+j)
&\sum\limits_{j=0}^{-m_N-2}F_{X_N}(m_N+j)
&\ldots
&\sum\limits_{j=0}^{1}F_{X_N}(m_N+j)
&p^{(N)}_{m_N}
\end{pmatrix},
\end{align*}
\begin{align*}
&\left(A_2\right)_{(-D)\times (-m_1)}:=\\
&\hspace{-1cm}\begin{pmatrix}
\sum\limits_{j=0}^{-m_1-1}\al_1^jF_{X_1}(m_1+j)
&\sum\limits_{j=0}^{-m_1-2}\al_1^{j+1}F_{X_1}(m_1+j)
&\ldots
&\sum\limits_{j=0}^{1}\al_1^{j-m_1-2}F_{X_1}(m_1+j)
&\al_1^{-m_1-1}p^{(1)}_{m_1}\\
\vdots&\vdots&\ddots&\vdots&\vdots\\
\sum\limits_{j=0}^{-m_1-1}\al_{C-1}^jF_{X_1}(m_1+j)
&\sum\limits_{j=0}^{-m_1-2}\al_{C-1}^{j+1}F_{X_1}(m_1+j)
&\ldots
&\sum\limits_{j=0}^{1}\al_{C-1}^{j-m_1-2}F_{X_1}(m_1+j)
&\al_{C-1}^{-m_1-1}p^{(1)}_{m_1}\\
\sum\limits_{j=0}^{-m_1-1}F_{X_1}(m_1+j)
&\sum\limits_{j=0}^{-m_1-2}F_{X_1}(m_1+j)
&\ldots
&\sum\limits_{j=0}^{1}F_{X_1}(m_1+j)
&p^{(1)}_{m_1}
\end{pmatrix},
\end{align*}
$$
\vdots
$$
\begin{align*}
&\left(A_N\right)_{(-D)\times (-m_{N-1})}:=\\
&\hspace{-1cm}\begin{pmatrix}
\sum\limits_{j=0}^{-m_{N-1}-1}\al_1^jF_{X_{N-1}}(m_{N-1}+j)
&\sum\limits_{j=0}^{-m_{N-1}-2}\al_1^{j+1}F_{X_{N-1}}(m_{N-1}+j)
&\ldots
&\al_1^{-m_{N-1}-1}p^{(N-1)}_{m_{N-1}}\\
\vdots&\vdots&\ddots&\vdots\\
\sum\limits_{j=0}^{-m_{N-1}-1}\al_{C-1}^jF_{X_{N-1}}(m_{N-1}+j)
&\sum\limits_{j=0}^{-m_{N-1}-2}\al_{C-1}^{j+1}F_{X_{N-1}}(m_{N-1}+j)
&\ldots
&\al_{C-1}^{-m_{N-1}-1}p^{(N-1)}_{m_{N-1}}\\
\sum\limits_{j=0}^{-m_{N-1}-1}F_{X_{N-1}}(m_{N-1}+j)
&\sum\limits_{j=0}^{-m_{N-1}-2}F_{X_{N-1}}(m_{N-1}+j)
&\ldots
&p^{(N-1)}_{m_{N-1}}
\end{pmatrix},
\end{align*}

\begin{align*}
\left(B_1\right)_{(-D)\times (-m_N)}&:=
\begin{pmatrix}
\al_1^{m_N}&\ldots&\al_1^{m_N}\\
\vdots&\ddots&\vdots\\
\al_{D-1}^{m_N}&\ldots&\al_{D-1}^{m_N}\\
1&\ldots&1
\end{pmatrix},\\
\left(B_2\right)_{(-D)\times (-m_1)}&:=
\begin{pmatrix}
\al_1^{m_1}G_{X_N}(\al_1)&\ldots&\al_1^{m_1}G_{X_N}(\al_1)\\
\vdots&\ddots&\vdots\\
\al_{D-1}^{m_1}G_{X_N}(\al_{D-1})&\ldots&\al_{D-1}^{m_1}G_{X_N}(\al_{D-1})\\\\
1&\ldots&1
\end{pmatrix},\\
\left(B_3\right)_{(-D)\times (-m_2)}&:=
\begin{pmatrix}
\al_1^{m_2}G_{X_N+X_1}(\al_1)&\ldots&\al_1^{m_2}G_{X_N+X_1}(\al_1)\\
\vdots&\ddots&\vdots\\
\al_{D-1}^{m_2}G_{X_N+X_1}(\al_{D-1})&\ldots&\al_{D-1}^{m_2}G_{X_N+X_1}(\al_{D-1})\\
1&\ldots&1
\end{pmatrix},\,\ldots,\\
\left(B_N\right)_{(-D)\times (-m_{N-1})}&:=
\begin{pmatrix}
\al_1^{m_{N-1}}G_{X_N+X_1+\ldots+X_{N-2}}(\al_1)&\ldots&\al_1^{m_{N-1}}G_{X_N+X_1+\ldots+X_{N-2}}(\al_1)\\
\vdots&\ddots&\vdots\\
\al_{D-1}^{m_{N-1}}G_{X_N+X_1+\ldots+X_{N-2}}(\al_{D-1})&\ldots&\al_{D-1}^{m_{N-1}}G_{X_N+X_1+\ldots+X_{N-2}}(\al_{D-1})\\
1&\ldots&1
\end{pmatrix}.
\end{align*}

{\sc Note 1.} {\it In case $G_{S_N}(s)=1$ has multiple roots in $|s|\leqslant1$, $s\neq1$, we augment the system \eqref{main_system} by including derivatives of \eqref{eq:memory_compact} evaluated at these roots up to one order less than their multiplicities.}

The system \eqref{main_system} determines the unknown initial probabilities provided its coefficient matrix is non-singular. Extensive numerical experiments suggest that this is always the case whenever $\mathbb{E}S_N<0$. 

In the special case when $b_k(x)=0$ for all $k\leqslant x$, $x\in\mathbb{Z}$ in \eqref{eq:recurrence_with_B}, the coefficient matrix in \eqref{main_system} reduces to a Vandermonde matrix. Consequently, its determinant admits the classical product representation in terms of the differences of the zeros of $G_{S_N}(s)-1$ and is therefore nonzero whenever these zeros are distinct. This leads to the explicit solution obtained in \cite{Grigutis2024}. However, a general non-singularity proof when $b_k(x)\neq0$, $k\leqslant x$, $x\in\mathbb{Z}$ in \eqref{eq:recurrence_with_B} remains elusive.

\begin{conjecture}
The coefficient matrix in \eqref{main_system} is non-singular whenever $\mathbb{E}S_N<0$.
\end{conjecture}

The coefficient matrix \eqref{main_system} may be viewed as a generalized Vandermonde-type matrix, as its rows are obtained by evaluating certain functions at the zeros of $G_{S_N}(s)-1$. This suggests that the non-singularity question may be related to interpolation or algebraic independence properties of these zeros.

In this work, we do not consider the distribution functions
$F_{\infty}^{(1)},\,F_{\infty}^{(2)},\,\ldots,\,F_{\infty}^{(N)}$
when $\mathbb{E}S_N\geqslant0$. It is known that then these functions are identically zero with some exceptions for the degenerate case $\mathbb{P}(S_N=0)=1$; see \cite[Lem.~4]{Grigutis2024}. 

Although Theorem \ref{T2} in this work provides a general computational procedure for the distribution functions $F_{\infty}^{(1)},\,F_{\infty}^{(2)},\,\ldots,\,F_{\infty}^{(N)}$, Proposition \ref{T1} may yield completely explicit distributions and moment formulas when the generating functions are rational and admit a favorable factorization; see Section \ref{sec:Examples}.

\section{Proofs}\label{sec:proofs}
In this section,  we prove all statements formulated in Section \ref{sec:results}.

\begin{proof}{of Proposition \ref{T1}.}
Let $a^+=\max\{0,\,a\},\,a\in\mathbb{R}$ be the positive part function. Then, for the random variables $\mathcal{M}_1,\,\mathcal{M}_2,\,\ldots,\,\mathcal{M}_N$ it holds that 
\begin{align}\label{syst:dist_eq}
(\mathcal M_{j+1}+X_j)^+
&\dist
\max\Bigl\{0,\,
X_j,\,
X_j+X_{j+1},\,
X_j+X_{j+1}+X_{j+2},\,
\ldots
\Bigr\}
\dist
\mathcal M_j,
\quad j=1,\,2,\,\ldots,\,N,
\end{align}
where $\mathcal{M}_{N+1}\dist\mathcal{M}_1$.

Identities \eqref{syst:dist_eq} imply

\begin{align}\label{syst:gf_eq}
G_{(\mathcal M_{j+1}+X_j)^+}(s)
=
G_{\mathcal M_j}(s),
\qquad j=1,\,2,\,\ldots,\,N.
\end{align}

From \eqref{syst:gf_eq}, for each $j=1,\,2,\,\ldots,\,N$, we obtain

\begin{align}\label{rearangements} \nonumber 
&G_{\mathcal{M}_j}(s)
=G_{(\mathcal{M}_{j+1}+X_j)^+}(s)
=\mathbb{E}\left(s^{(\mathcal{M}_{j+1}+X_j)^+}\right) =\mathbb{E}\left(\mathbb{E}\left(s^{(\mathcal{M}_{j+1}+X_j)^+}|\mathcal{M}_{j+1}\right)\right)\\ \nonumber 
&=\sum_{i=0}^{-m_j-1}\pi^{(j+1)}_i\mathbb{E}s^{\left(i+X_j\right)^+} 
+\sum_{i=-m_j}^{\infty}\pi^{(j+1)}_i\mathbb{E}s^{i+X_j}\\ \nonumber 
&=\sum_{i=0}^{-m_j-1}\pi^{(j+1)}_i \left(\sum_{k=m_j}^{-i-1}p^{(j)}_k +\sum_{k=-i}^{\infty}s^{i+k}p_k^{(j)} \right)
+G_{X_j}(s)\left(G_{\mathcal{M}_{j+1}}(s)-\sum_{i=0}^{-m_j-1}\pi^{(j+1)}_i s^i \right)\\ \nonumber &=\sum_{i=0}^{-m_j-1}\pi^{(j+1)}_i \left(\sum_{k=m_j}^{-i-1}p^{(j)}_k +\sum_{k=-i}^{\infty}s^{i+k}p_k^{(j)} \right)+G_{X_j}(s)G_{\mathcal{M}_{j+1}}(s)\\ \nonumber 
&-\sum_{i=0}^{-m_j-1}\pi^{(j+1)}_i\sum_{k=m_j}^{-i-1}p^{(j)}_k s^{i+k} -\sum_{i=0}^{-m_j-1}\pi^{(j+1)}_i\sum_{k=-i}^{\infty}p^{(j)}_k s^{i+k}\\
&=
G_{X_j}(s)G_{\mathcal M_{j+1}}(s)
+
\sum_{i=0}^{-m_j-1}\pi_i^{(j+1)}
\left(
F_{X_j}(-i-1)
-\sum_{k=m_j}^{-i-1}p_k^{(j)}s^{i+k}
\right)\nonumber\\
&=
G_{X_j}(s)G_{\mathcal M_{j+1}}(s)
-
(1-s)s^{m_j}
\sum_{i=0}^{-m_j-1}\pi_i^{(j+1)}
\sum_{r=0}^{-m_j-1-i}
s^{r+i}F_{X_j}(m_j+r),
\end{align}

where the last identity was obtained by
\begin{align*}
\frac{1-s^{i+k}}{1-s}=-\sum_{\ell=i+k}^{-1}s^\ell,\,i+k\in\{-1,\,-2,\,\ldots\}
\end{align*}
and
\begin{align*}
\sum\limits_{k=m_j}^{-i-1}p_k^{(j)}\frac{1-s^{i+k}}{1-s}&=-
\sum\limits_{k=m_j}^{-i-1}p_k^{(j)}\sum_{\ell=i+k}^{-1}s^\ell
=-\sum_{\ell=m_j+i}^{-1}s^\ell\sum_{k=m_j}^{\ell-i}p_k^{(j)}
=-\sum_{\ell=m_j+i}^{-1}s^\ell F_{X_j}(\ell-i)\\
&=-s^{m_j}\sum_{r=0}^{-m_j-i-1}s^{i+r}F_{X_j}(m_j+r).
\end{align*}

For each $j=1,\,2,\,\ldots,\,N$, the matrix form of \eqref{rearangements} is 
\begin{align}\label{matrix_identity}
A(s){\pmb G}_{\mathcal M}(s)
=(s-1){\pmb B}(s),
\end{align}
where 
\begin{align*}
{\pmb G}_{\mathcal M}(s)
&=
\left(
G_{\mathcal M_1}(s),\,
G_{\mathcal M_2}(s),\,
\ldots,\,
G_{\mathcal M_N}(s)
\right)^T,\\
{\pmb B}(s)&=
\left(
B_1(s),\,
B_2(s),\,
\ldots,\,
B_N(s)
\right)^T,
\end{align*}
with

\begin{align*}
B_j(s)&=s^{m_j}\sum\limits_{i=0}^{-m_j-1}\pi^{(j+1)}_i\,s^i\sum\limits_{k=0}^{-m_j-1-i}s^kF_{X_j}(m_j+k)\\
&=s^{m_j}\sum_{i=0}^{-m_j-1}\pi^{(j+1)}_i s^i\sum_{n=i}^{-m_j-1}s^{n-i}F_{X_j}(m_j+n-i)\\
&=\sum_{n=0}^{-m_j-1}s^{n+m_j}\sum_{i=0}^{n}\pi_i^{(j+1)}F_{X_j}(m_j+n-i)\\
&=\sum_{k=m_j}^{-1}s^k\sum_{i=0}^{k-m_j}\pi_i^{(j+1)}F_{X_j}(k-i),
\quad j=1,\,2,\,\ldots,\,N,
\end{align*}
and 
\begin{align*}
A(s)
=
\begin{pmatrix}
1 & -G_{X_1}(s) & 0 & \cdots & 0\\
0 & 1 & -G_{X_2}(s) & \cdots & 0\\
\vdots & \vdots & \vdots & \ddots & \vdots\\
0 & 0 & 0 & \cdots & -G_{X_{N-1}}(s)\\
-G_{X_N}(s) & 0 & 0 & \cdots & 1
\end{pmatrix}.
\end{align*}
We now compute
\begin{align*}
&A^{-1}(s)=
\frac{1}{1-G_{S_N}(s)}\times\\
&\begin{pmatrix}\nonumber
1&G_{X_1}(s)&G_{X_1+X_2}(s)&\ldots&G_{X_1+X_2+\ldots+X_{N-1}}(s)\\
G_{X_2+X_3+\ldots+X_N}(s)&1&G_{X_2}(s)&\ldots&
G_{X_2+X_3+\ldots+X_{N-1}}(s)\\
G_{X_3+X_4+\ldots+X_N}(s)&G_{X_1+X_3+X_4+\ldots+X_N}(s)&1&\ldots&G_{X_3+X_4+\ldots+X_{N-1}}(s)\\
\vdots&\vdots&\vdots&\ddots&\vdots\\
G_{X_{N-1}+X_N}(s)&G_{X_1+X_{N-1}+X_N}(s)&G_{X_1+X_2+X_{N-1}+X_N}(s)&\ldots&G_{X_{N-1}}(s)\\
G_{X_N}(s)&G_{X_1+X_N}(s)&G_{X_1+X_2+X_N}(s)&\ldots&1
\end{pmatrix}.
\end{align*}
It is straightforward to verify $A(s)A^{-1}(s)={\pmb I}$, where ${\pmb I}$ is the identity matrix. This completes the proof.
\hfill $\square$
\end{proof}

\bigskip

\begin{proof}{of Theorem \ref{T2}.}
Let us begin by expanding $1/(G_{S_N}(s)-1)$ by Maclaurin series. For $s\in\mathbb{C}$ such that $|s|<1$ and $G_{S_N}(s)\neq1$, we have
\begin{align*}
\frac{1}{G_{S_N}(s)-1}=\frac{s^{-D}}{\sum\limits_{n=D}^{\infty}s^{n-D}\,f_N(n)-s^{-D}}=\sum_{n=-D}^{\infty}s^n\,a_n,\,D=m_1+m_2+\cdots+m_N<0.
\end{align*}
From this
\begin{align*}
1=\sum_{n=-D}^{\infty}s^n\,a_n\left(\sum_{n=D}^{\infty}s^n\,f_N(n)-1\right)=\sum_{n=0}^{\infty}\left(\sum_{k=-D}^{n-D}a_kf_N(n-k)\right)\,s^n-\sum_{n=-D}^{\infty}s^n\,a_n
\end{align*}
and consequently
\begin{align}\label{coeff}
1+\sum_{n=-D}^{\infty}s^n\,a_n=
\sum_{n=0}^{\infty}\left(\sum_{k=-D}^{n-D}a_kf_N(n-k)\right)\,s^n,\quad -D>0.
\end{align}
By equating the coefficients at the powers of $s$ in \eqref{coeff}, we obtain
\begin{align}\label{the_claimed}
a_{-D}=\frac{1}{f_N(D)}, \, a_{n}=\frac{1}{f_N(D)}\left({\pmb 1}_{\{n\geqslant-2D\}}a_{n+D}-\sum_{k=1}^{n+D}f_N(k+D)a_{n-k}\right),\quad n\geqslant-D+1.
\end{align}
Indeed, by looking at the right-hand side of \eqref{coeff}, we define
\begin{align}\label{defined_R}
R_n:=\sum_{k=-D}^{n-D}a_kf_N(n-k),\quad n\in\mathbb{N}_0.
\end{align}

We shall show
\begin{align}\label{statement}
R_n=
\begin{cases}
1,&n=0,\\
0,&1\leqslant n <-D,\\
a_n,&n\geqslant -D.
\end{cases}
\end{align}

If $n=0$, then by \eqref{coeff}
\begin{align*}
1=R_0=a_{-D}f_N(D).
\end{align*}
If $n\in\mathbb{N}$, then from \eqref{defined_R} and \eqref{the_claimed}
\begin{align*}
R_n&=a_{n-D}f_N(D)+\sum_{k=-D}^{n-D-1}a_kf_N(n-k)\\
&={\pmb 1}_{\{n\geqslant -D\}}a_n-\sum_{k=1}^{n}f_N(k+D)a_{n-D-k}
+\sum_{k=-D}^{n-D-1}a_kf_N(n-k)\\
&={\pmb 1}_{\{n\geqslant -D\}}a_n,
\end{align*}
which yields \eqref{statement}.

Thus, having the Maclaurin series of $1/(G_{S_N}(s)-1)$, we proceed comparing the coefficients at the powers of $s$ in \eqref{eq:T1}. Recall that

\begin{align*}
\frac{G_{\mathcal{M}_j}(s)}{1-s}=\sum_{n=0}^{\infty}F^{(j)}_{\infty}(n)s^n,\quad |s|<1,\quad j=1,\,2,\,\ldots,\,N.
\end{align*}

By Proposition~\ref{T1},
\begin{align*}
\frac{G_{\mathcal M_j}(s)}{1-s}
&=\frac{1}{G_{S_N}(s)-1}
\sum_{r=1}^{N}B_r(s)\,G_{\sum_{k\in I_{jr}}X_k}(s),
\qquad j=1,\ldots,N,
\end{align*}
where
\begin{align*}
B_r(s):=\sum_{\ell=m_r}^{-1}s^\ell\sum_{i=0}^{\ell-m_r}\pi^{(r+1)}_iF_{X_r}(\ell-i),\quad r=1,\,2,\,\ldots,\,N,
\end{align*}
and
\[
I_{jr}:=
\begin{cases}
(j,j+1,\ldots,r-1), & j<r,\\
\varnothing, & j=r,\\
(1,\,\ldots,\,r-1,\,j,\,j+1,\,\ldots,\,N), & j>r.
\end{cases}
\]
with then convention that $G_{\varnothing}=1$.
Then
\begin{align}\label{to_compare}
\sum_{n=0}^{\infty}F^{(j)}_{\infty}(n)s^n
=\sum_{r=1}^{N}
\left(
\sum_{n=0}^{\infty}
a_n^{I_{jr}}s^n
\right)
\left(
\sum_{\ell=m_r}^{-1}
s^\ell
\sum_{i=0}^{\ell-m_r}
\pi_i^{(r+1)}
F_{X_r}(\ell-i)
\right)
\end{align}
since
\begin{align*}
\frac{G_{\sum_{k\in I_{jr}}X_k}(s)}{G_{S_N}(s)-1}=\sum_{n=0}^{\infty}
a_n^{I_{jr}}s^n,\quad j,r=1,\,2,\,\ldots,\,N,
\end{align*}
where 
\begin{align*}
a_n^{I_{jr}}
=\sum_{k=\mu_{jr}}^{n+D}
p_k^{I_{jr}}
a_{n-k},
\qquad
\mu_{jr}=\sum_{t\in I_{jr}}m_t,
\qquad
n\geqslant \mu_{jr}-D,
\end{align*}

\begin{align*}
p_k^{I_{jr}}
=
\mathbb{P}\!\left(
\sum_{t\in I_{jr}} X_t = k
\right),
\qquad
k\in\mathbb{Z},
\qquad
p_k^{\varnothing}
=
\begin{cases}
1, & k=0,\\
0, & k\neq 0,
\end{cases}
\end{align*}

Since $a_n^{I_{jr}}=0$ for $n<\mu_{jr}-D$, the right-hand side of \eqref{to_compare} contains only nonnegative powers of $s$. Therefore, comparing the coefficients of $s^n$ in \eqref{to_compare} yields
\[
F_\infty^{(j)}(n)
=
\sum_{r=1}^{N}
\sum_{\ell=m_r}^{-1}
a_{n-\ell}^{I_{jr}}
\sum_{i=0}^{\ell-m_r}
\pi_i^{(r+1)}
F_{X_r}(\ell-i),
\]
which is the assertion of the theorem.
\hfill $\square$
\end{proof}

\bigskip

\begin{proof}{of Proposition \ref{T3}.}
We first prove \eqref{eq1_T3}. By the definition, for each $j=1,\,2,\,\ldots,\,N$,
\begin{align*}
F_\infty^{(j)}(n)
=
\mathbb P\left(
X_j\leqslant n,\,
X_j+X_{j+1}\leqslant n,\,
X_j+X_{j+1}+X_{j+2}\leqslant n,\,
\ldots
\right),\quad n\geqslant M_j.
\end{align*}
Conditioning on the value of $X_j$, we obtain
\begin{align*}
F_\infty^{(j)}(n)
&=
\sum_{k=m_j}^{\infty}
\mathbb P(X_j=k)\,
\mathbb P\left(
k+X_{j+1}\leqslant n,\,
k+X_{j+1}+X_{j+2}\leqslant n,\,
\ldots
\right)\\
&=
\sum_{k=m_j}^{\infty}
p_k^{(j)}
\mathbb P\left(
X_{j+1}\leqslant n-k,\,
X_{j+1}+X_{j+2}\leqslant n-k,\,
\ldots
\right)\\
&=
\sum_{k=m_j}^{\infty}
p_k^{(j)}
F_\infty^{(j+1)}(n-k).
\end{align*}
We now prove \eqref{eq2_T3}. Observe that
\begin{align*}
F_\infty^{(j)}(m_j+n)
&=
\mathbb P\left(
X_j\leqslant m_j+n,\,
X_j+X_{j+1}\leqslant m_j+n,\,
X_j+X_{j+1}+X_{j+2}\leqslant m_j+n,\,
\ldots
\right)\\
&=
\mathbb P\left(
X_j+\mathcal M_{j+1}\leqslant m_j+n
\right),
\end{align*}
where
\[
\mathcal M_{j+1}
=
\max\{0,\,X_{j+1},\,X_{j+1}+X_{j+2},\,\ldots\}.
\]
Conditioning on $\mathcal M_{j+1}$, we obtain
\begin{align*}
F_\infty^{(j)}(m_j+n)
=
\sum_{k=0}^{\infty}
\mathbb P(\mathcal M_{j+1}=k)
\mathbb P(X_j\leqslant m_j+n-k)
=
\sum_{k=0}^{\infty}
\pi_k^{(j+1)}
F_{X_j}(m_j+n-k),\quad n\geqslant0.
\end{align*}
\hfill $\square$
\end{proof}

\begin{proof}{of Lemma 1.}

The identity \eqref{eq:memory_compact} follows by multiplying both sides of \eqref{matrix_identity} by the row vector
\begin{align}\label{vector}
\left(
G_{X_N}(s),\,
G_{X_N+X_1}(s),\,
\ldots,\,
G_{X_N+X_1+\cdots+X_{N-2}}(s),\,
1
\right).
\end{align}
Notice that cyclic permutations of the entries of the row vector \eqref{vector} yield analogous identities to \eqref{eq:memory_compact}. However, these identities are equivalent, since the corresponding row vectors differ only by multiplication by suitable products of the generating functions $G_{X_j}(s)$.

Identity \eqref{eq:memory_compact_E} follows from \eqref{eq:memory_compact} by differentiating both sides and letting $s\to1^{-}$; see \cite[pp.~340--341]{gerve_grigutis_2024}.
\hfill $\square$
\end{proof}
\newpage
\section{Examples}\label{sec:Examples}
In this section, we provide several examples verifying the correctness of the derived formulas. The presented computations and visualizations are implemented using the software \cite{Python2} and \cite{Mathematica}. We first revisit Example~2 from \cite{Grigutis2024}.

\begin{ex}\label{ex_1}(\cite[Ex. 2]{Grigutis2024})
Suppose $X_1$, $X_2$, and $X_3$ are independent random variables such that $X_i\dist X_{i+3}$ for all $i\in\mathbb{N}$ and their probability distributions are
\begin{align*}
p^{(1)}_k=\mathbb{P}(X_1=k)
&=0.55(0.45)^{k-1}, &\quad k&=1,\,2,\,\ldots,\\
p^{(2)}_k=\mathbb{P}(X_2=k)
&=e^{-1/2}\frac{(1/2)^{k+3}}{(k+3)!}, &\quad k&=-3,\,-2,\,\ldots,\\
p^{(3)}_k=\mathbb{P}(X_3=k)
&=e^{-k}-e^{-k-1}, &\quad k&=0,\,1,\,\ldots.
\end{align*}
We compute the selected values of the distribution functions 
\begin{align*}
&F^{(1)}_{\infty}(n)=\mathbb{P}(X_1\leqslant n,\,X_1+X_2\leqslant n,\,\ldots),\quad n\geqslant M_1,\\
&F^{(2)}_{\infty}(n)=\mathbb{P}(X_2\leqslant n,\,X_2+X_3\leqslant n,\,\ldots),\quad n\geqslant M_2,\\
&F^{(3)}_{\infty}(n)=\mathbb{P}(X_3\leqslant n,\,X_3+X_4\leqslant n,\,\ldots),\quad n\geqslant M_3, 
\end{align*}
where
\begin{align*}
M_1&=\max\{1,\,1-3,\,1-3+0\}=1,\\
M_2&=\max\{-3,\,-3+0,\,-3+0+1\}=-2,\\
M_3&=\max\{0,\,0+1,\,0+1-3\}=1,
\end{align*}
by statements given in Section \ref{sec:results}, and for the sake of correctness, we approximate these distribution functions by computer simulations also.
\end{ex}

We begin with the mentioned computer simulations that approximate $F_{\infty}^{(1)},\,F_{\infty}^{(2)},$ and $F_{\infty}^{(3)}$. Let us describe our experiment. First, we generate $T\in\mathbb N$ independent realizations of the random variables
$X_1,\,X_2,\,\ldots,\,X_M$, $M\in\mathbb N$, where the random variables in the sequence are independent and satisfy $X_i\dist X_{i+3}$ for all $i\in\mathbb N$. For this, we use the Python library
{\sc NumPy} \cite{numpy} and the probability mass functions
$p^{(1)},\,p^{(2)},\,p^{(3)}$. For each realization, we construct the array of partial sums
\begin{align*}
P_i^{(j)}
:=
\left(
S_{i,\,1}^{(j)},\,S_{i,\,2}^{(j)},\,\ldots,\,S_{i,\,M-j+1}^{(j)}
\right),\quad j=1,\,2,\,3,
\end{align*}
where
\begin{align*}
S_{i,\,m}^{(j)}
=
X_j^{(i)}+X_{j+1}^{(i)}+\cdots+X_{j+m-1}^{(i)},
\quad
m=1,\,2,\,\ldots,\,M-j+1, 
\end{align*}
and the index $i=1,\,2,\,\ldots,\,T$ refers to the $i$'th realization. Thus, each array $P_i^{(j)}$ contains $M-j+1$ partial sums beginning with $X_j^{(i)}$. Fix $n\in\mathbb{Z}$. We classify the array (also called trajectory) $P_i^{(j)}$ as \emph{green} if all of its elements are
$\leqslant n$, and classify it as \emph{red} otherwise. Equivalently, a trajectory is colored red whenever it crosses the threshold level $n$ at least once within the observation window from $1$ to $M-j+1$. Figure~\ref{antra_figura} depicts sample trajectories for $j=1$,
$n=1,\,3,\,5,\,10$, $M=250$, and $T=800$.

\begin{figure}[H]
\centering
 \includegraphics[height=5cm,width=\textwidth]{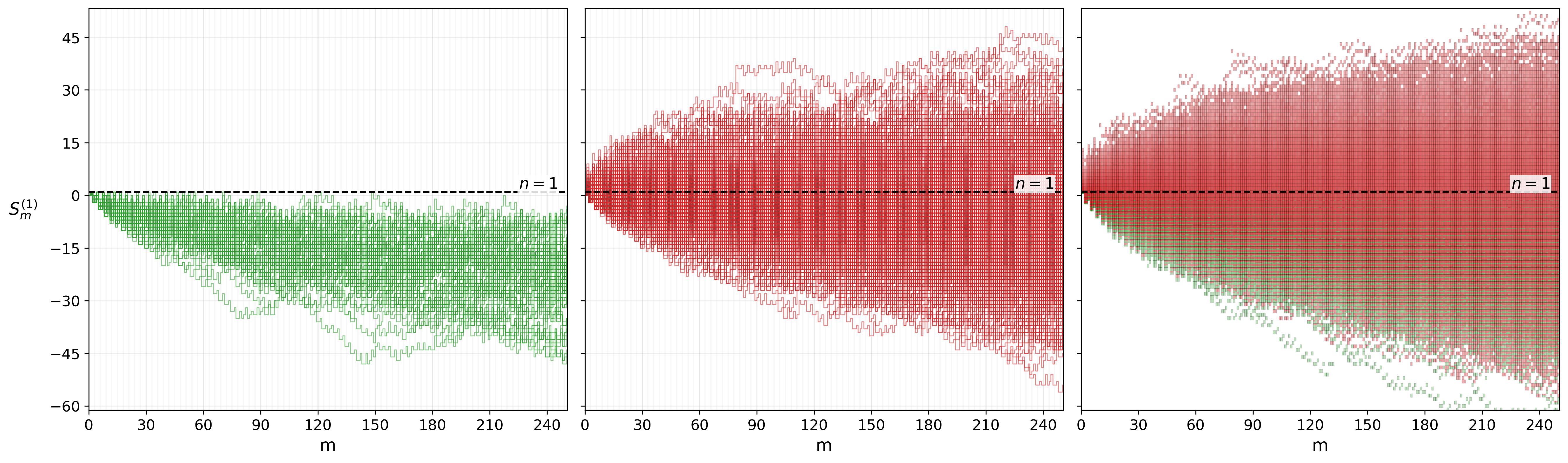}
 \includegraphics[height=5cm,width=\textwidth]{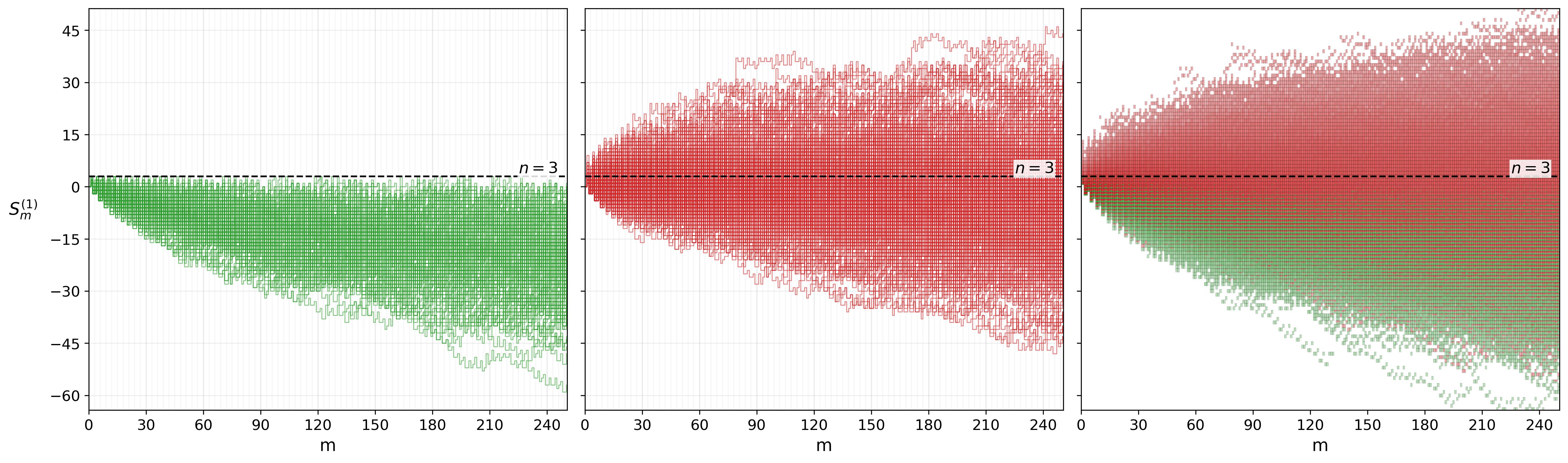}
  \includegraphics[height=5cm,width=\textwidth]{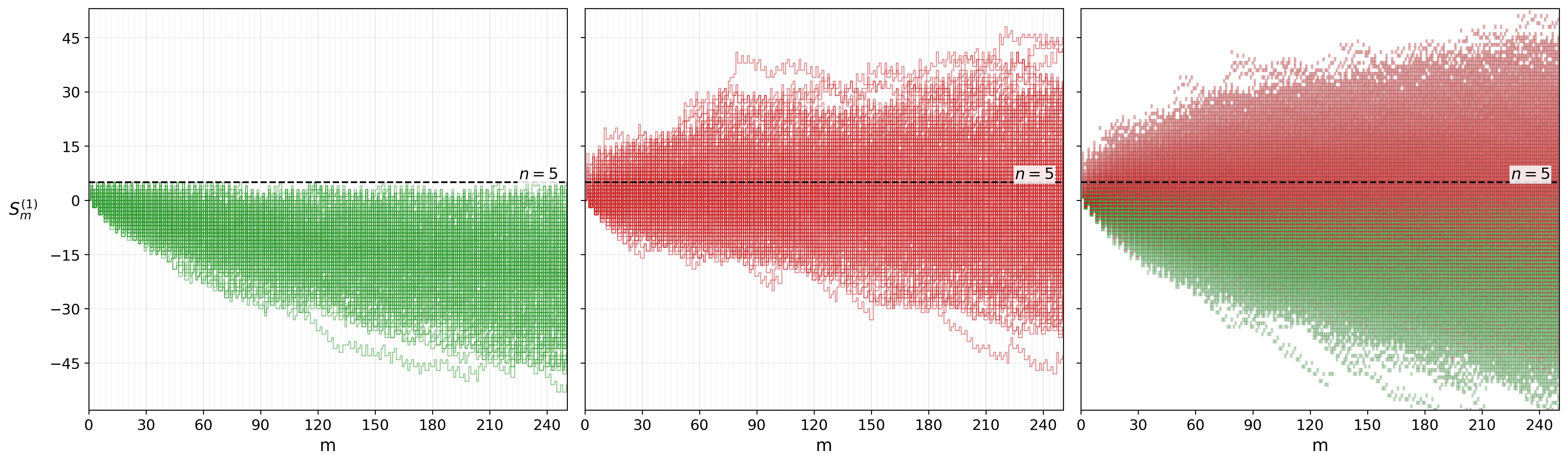}
 \includegraphics[height=5cm,width=\textwidth]{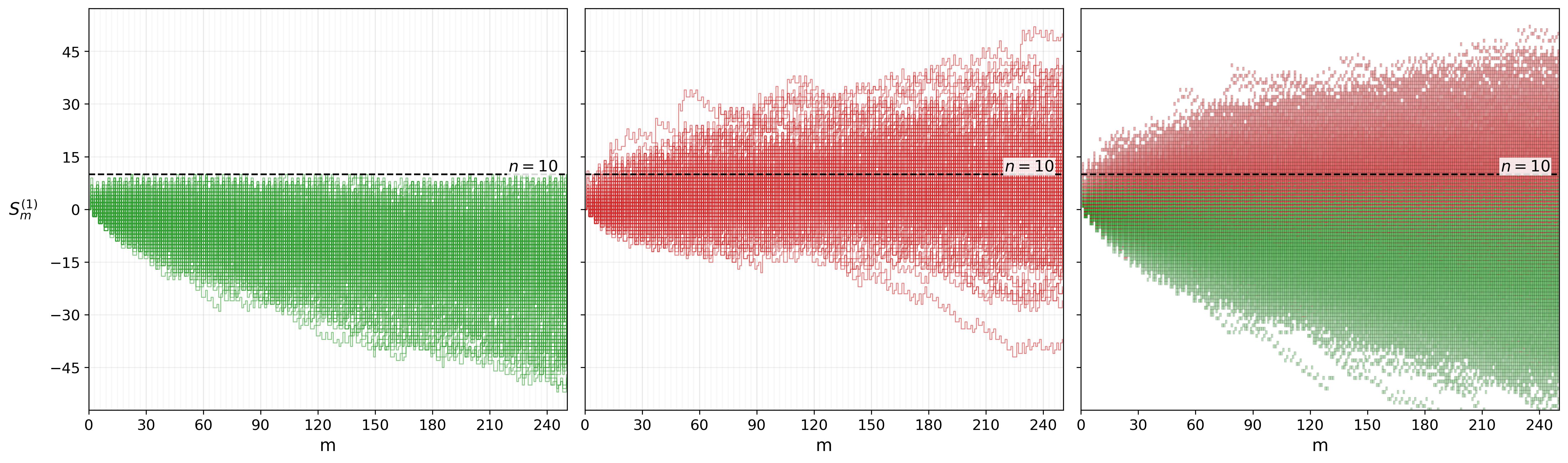}
 \caption{Sample trajectories (left -- green, middle -- red, right -- both) corresponding to the arrays $P_i^{(1)}$ for $n=1,\,3,\,5,\,10$, $M=250$, and $T=800$. Green trajectories do not cross the threshold level $n$, while red trajectories cross it at least once.
}\label{antra_figura}
\end{figure}

As seen in Figure \ref{antra_figura}, the number of green trajectories increases with $n$, while the number of red trajectories decreases. Hence, for each $j=1,\,2,\,3$, the proportion of green trajectories
\begin{align*}
\hat F_\infty^{(j)}(n)
:=
\frac{1}{T}
\#
\left\{
i\in\{1,\ldots,T\}:
\max P_i^{(j)}
\leqslant n
\right\},
\end{align*}
is expected to approximate the distribution function
$F_\infty^{(j)}(n)$. More extensive results of the
described experiment are presented in Table~\ref{tab:numerical_values}.

\begin{table}[h]
\centering
\captionsetup{justification=centering}
\begin{tabular}{c|ccccccc}
\hline
$n$ & 1 & 2 & 3 & 5 & 10 & 20 & 30 \\
\hline
$\hat{F}_{\infty}^{(1)}(n)$ & 0.0693 & 0.1303 & 0.1863 & 0.2872 & 0.4856 & 0.7313 & 0.8604\\
\hline
$\hat{F}_{\infty}^{(2)}(n)$ & 0.1804 & 0.2309 & 0.2794 & 0.3679 & 0.5437 & 0.7630 & 0.8767\\
\hline
$\hat{F}_{\infty}^{(3)}(n)$ & 0.0435 & 0.0983 & 0.1536 & 0.2572 & 0.4643 & 0.7203 & 0.8545\\
\hline
\end{tabular}
\caption{Monte Carlo estimates of the distribution functions
$F_{\infty}^{(1)},\,F_{\infty}^{(2)},\,F_{\infty}^{(3)}$
obtained with $M=300\,000$ and $T=50\,000$. The values are rounded to four decimal places.}
\label{tab:numerical_values}
\end{table}

We now proceed with the analytical solution based on the statements in Section \ref{sec:results}. Let us observe that in this example,
\begin{align*}
&\mathbb{P}\left(\max\{X_1,\,X_1+X_2,\,X_1+X_2+X_3\}= X_1+X_2+X_3\right)<1,\\
&\mathbb{P}\left(\max\{X_2,\,X_2+X_3,\,X_2+X_3+X_1\}=X_2+X_3+X_1\right)=1,\\
&\mathbb{P}\left(\max\{X_3,\,X_3+X_1,\,X_3+X_1+X_2\}= X_3+X_1+X_2\right)<1,
\end{align*}
where the second equality follows from the inequalities
\begin{align}\label{reason}
\mathbb{P}\left(X_2+X_3+X_1\geqslant X_2+X_3\geqslant X_2\right)=1
\end{align}
which hold because $X_3\geqslant 0$, $X_1\geqslant 1$ a.s.
Therefore, according to what was explained in subsection \ref{subsec:prior}, only $F_{\infty}^{(2)}$ from this example falls within the scope of \cite[Thm. 1]{Grigutis2024}. Moreover, by \eqref{reason},
\begin{align*}
\max\{Z_1,\,Z_1+Z_2,\,\ldots\}\dist\max\{X_2,\,X_2+X_3,\,X_2+X_3+X_1,\,\ldots\},
\end{align*}
where $Z_1,\,Z_2,\,\ldots$ are independent copies of $X_1+X_2+X_3$.

In this example 
\begin{align*}
\mathbb{E}S_3=\frac{1}{e-1}-\frac{15}{22}\approx -0.0998<0
\end{align*}
and
\begin{align*}
G_{S_3}(s)=G_{X_1}(s)G_{X_2}(s)G_{X_3}(s)=\frac{0.55}{1-0.45s}\frac{e^{-1/2+s/2}}{s^2}\frac{e-1}{e-s}=1
\end{align*}
has a root $\alpha:\approx-0.364796$ inside the unit disk. We now need to set up the system \eqref{main_system}. Since $m_1=1$, $m_2=-3$, $m_3=0$, and
\begin{align*}
\pi_0^{(3)}=\mathbb{P}(\max\{0,\,X_3,\,X_3+X_1,\,\ldots\}=0)=0
\end{align*}
as $\mathbb{P}(X_3+X_1\geqslant1)=1$, Lemma \ref{L1} yields $2\times2$ linear system

\begin{align*}
\begin{pmatrix}
p_{-3}^{(2)}+\alpha F_{X_2}(-2) &
\alpha p_{-3}^{(2)}
\\[2mm]
p_{-3}^{(2)}+F_{X_2}(-2) &
p_{-3}^{(2)}
\end{pmatrix}
\begin{pmatrix}
\pi_1^{(3)}\\
\pi_2^{(3)}
\end{pmatrix}
=
\begin{pmatrix}
0\\
-\mathbb{E}S_3
\end{pmatrix},
\end{align*}
whose solution is

\begin{align}
\pi_1^{(3)}\label{sol1}
&= \frac{\mathbb{E}S_3}{p_{-3}^{(2)}} \frac{\alpha}{1-\alpha}=
\sqrt{e}\left(\frac{1}{e-1}-\frac{15}{22}\right)\frac{\alpha}{1-\alpha}\approx0.0439987,
\\
\pi_2^{(3)}\label{sol2}
&= -\frac{\mathbb{E}S_3}{\left(p_{-3}^{(2)}\right)^2}
\frac{p_{-3}^{(2)}+\alpha F_{X_2}(-2)}{1-\alpha}
=-\sqrt{e}\left(\frac{1}{e-1}-\frac{15}{22}\right)\frac{1+3\alpha/2}{1-\alpha}\approx
0.0546139.
\end{align}

We first obtain $F^{(1)}_{\infty}$. According to Theorem \ref{T2} and having that $m_1=1$, $m_2=-3$, $m_3=0$, $\pi^{(3)}_0=0$, only $j=2$ term contributes and we obtain
\begin{align*}
F_{\infty}^{(1)}(n)
&=\sum_{j=1}^{3}\sum_{\ell=m_j}^{-1}a^{I_{1j}}_{n-\ell}\sum_{i=0}^{\ell-m_j}\pi_i^{(j+1)}F_{X_j}(\ell-i)
=\sum_{\ell=-3}^{-1}a^{I_{12}}_{n-\ell}\sum_{i=0}^{\ell+3}\pi_i^{(3)}F_{X_2}(\ell-i)\\
&=a^{(1)}_{n+2}\pi_1^{(3)}p^{(2)}_{-3}+a_{n+1}^{(1)}\left(\pi^{(3)}_1\left(p^{(2)}_{-3}+p^{(2)}_{-2}\right)+\pi_2^{(3)}p^{(2)}_{-3}\right),\quad n\in\mathbb{N}_0,
\end{align*}
where $a_n^{I_{12}}=a_n^{(1)}$,
\begin{align*}
a^{(1)}_1=a^{(1)}_2=0,\quad a_n^{(1)}=\sum_{k=1}^{n-2}p_k^{(1)}a_{n-k},\quad n\geqslant3,
\end{align*}
and 
\begin{align*}
a_2=\frac{1}{f_3(-2)},\,
a_{n}=\frac{1}{f_3(-2)}\left({\pmb 1}_{\{n\geqslant4\}}a_{n-2}-\sum_{k=1}^{n-2}f_3(k-2)a_{n-k}\right),\quad n\geqslant3.
\end{align*}
Therefore, $F_{\infty}^{(1)}(n)=0$ for all $n\leqslant0$ and

\begin{align}\label{main_ex_1}
F_\infty^{(1)}(n)
=
0.55\left(\frac{1}{e-1}-\frac{15}{22}\right)
\frac{1}{1-\alpha}
\sum_{k=1}^{n}(0.45)^{k-1}
\Bigl(
\alpha a_{n+2-k}
-
a_{n+1-k}
\Bigr),
\qquad n\in\mathbb N,
\end{align}
where
\begin{align}\label{seq_a}
a_1=0,\quad
a_2
=
\frac{e^{3/2}}{0.55(e-1)},
\quad
a_n
=
a_2
\left({\pmb 1}_{\{n\geqslant4\}}a_{n-2}-\sum_{k=1}^{n-2}f_3(k-2)a_{n-k}\right),
\quad n\geqslant3,
\end{align}
with
\begin{align}\label{pmf_S3}
f_3(k)
=
\frac{0.55(e-1)}{e^{3/2}}
\sum_{r=1}^{k+3}
(0.45)^{r-1}
\sum_{u=0}^{k-r+3}
\frac{(1/2)^u}
     {u!}e^{-(k-r-u+3)},
\qquad k\geqslant -2.
\end{align}
Let us compute $F_{\infty}^{(2)}$. By Theorem \ref{T2}, for $j=2$ only, we have $I_{22}=\varnothing$. Therefore $a^{I_{22}}_n=a_n$ and
\begin{align*}
F_{\infty}^{(2)}(n)
&=\sum_{j=1}^{3}\sum_{\ell=m_j}^{-1}a^{I_{2j}}_{n-\ell}\sum_{i=0}^{\ell-m_j}\pi_i^{(j+1)}F_{X_j}(\ell-i)
=\sum_{\ell=-3}^{-1}a^{I_{22}}_{n-\ell}\sum_{i=0}^{\ell+3}\pi_i^{(3)}F_{X_2}(\ell-i)\\
&=a_{n+2}\pi_1^{(3)}p^{(2)}_{-3}+a_{n+1}\left(\pi^{(3)}_1\left(p^{(2)}_{-3}+p^{(2)}_{-2}\right)+\pi_2^{(3)}p^{(2)}_{-3}\right),\quad n\in\mathbb{N}_0,
\end{align*}
which simplifies to
\begin{align}\label{main_ex_2}
F_\infty^{(2)}(n)=\left(\frac{1}{e-1}-\frac{15}{22}\right)\frac{1}{1-\alpha}\left(\alpha a_{n+2}-a_{n+1}\right),
\qquad n\in\mathbb N_0,
\end{align}
where the coefficients $a_n$ are defined in \eqref{seq_a}. Moreover, according to Proposition \ref{T3},
\begin{align*}
F_{\infty}^{(2)}(-1)&=\pi_0^{(3)}F_{X_2}(-1)+\pi_1^{(3)}F_{X_2}(-2)+\pi_2^{(3)}F_{X_2}(-3)=\pi_1^{(3)}\left(p^{(2)}_{-3}+p^{(2)}_{-2}\right)+\pi_2^{(3)}p^{(2)}_{-3}\\
&=-\left(\frac{1}{e-1}-\frac{15}{22}\right)\frac{1}{1-\alpha}\approx 0.0732,\\
F_{\infty}^{(2)}(-2)&=\pi_0^{(3)}F_{X_2}(-2)+\pi_1^{(3)}F_{X_2}(-3)=\pi_1^{(3)}p^{(2)}_{-3}=\left(\frac{1}{e-1}-\frac{15}{22}\right)\frac{\alpha}{1-\alpha}\approx 0.0267.
\end{align*}
Of course, $F_{\infty}^{(2)}(n)=0$ if $n<-2$.

Let's compute $F_{\infty}^{(3)}$. By Theorem \ref{T2}, for $j=2$, we have
$I_{32}=(1,\,3)$. Therefore
\begin{align*}
F_{\infty}^{(3)}(n)
&=\sum_{j=1}^{3}\sum_{\ell=m_j}^{-1}a^{I_{3j}}_{n-\ell}\sum_{i=0}^{\ell-m_j}\pi_i^{(j+1)}F_{X_j}(\ell-i)
=\sum_{\ell=-3}^{-1}a^{I_{32}}_{n-\ell}\sum_{i=0}^{\ell+3}\pi_i^{(3)}F_{X_2}(\ell-i)\\
&=a^{{(1,\,3)}}_{n+2}\pi_1^{(3)}p^{(2)}_{-3}+a_{n+1}^{{(1,3)}}\left(\pi^{(3)}_1\left(p^{(2)}_{-3}+p^{(2)}_{-2}\right)+\pi_2^{(3)}p^{(2)}_{-3}\right),\quad n\in\mathbb{N}_0.
\end{align*}
Substituting the values of $\pi_1^{(3)}$ and $\pi_2^{(3)}$ from \eqref{sol1} and $\eqref{sol2}$, and using
\begin{align*}
a_1^{(1,\,3)}=a_2^{(1,\,3)}=0,\quad a_n^{(1,\,3)}
&=\sum_{k=1}^{n-2}p_k^{(1,3)}a_{n-k},\quad n\geqslant 3,\\
p_k^{(1,\,3)}
&=\sum_{r=1}^{k}0.55(0.45)^{r-1}\left(e^{-(k-r)}-e^{-(k-r)-1}\right),\quad k=1,2,\ldots,
\end{align*}
we obtain $F^{(3)}_{\infty}(n)=0,\, n\leqslant0$, and
\begin{align}\label{main_ex_3}
F_\infty^{(3)}(n)
&=\left(\frac{1}{e-1}-\frac{15}{22}\right)\frac{1}{1-\alpha}\left(\alpha a_{n+2}^{{(1,\,3)}}-a_{n+1}^{{(1,\,3)}}\right)\nonumber\\
&=\left(\frac{1}{e-1}-\frac{15}{22}\right)\frac{1}{1-\alpha}\sum_{k=1}^{n}p_k^{{(1,\,3)}}\left(\alpha a_{n+2-k}-a_{n+1-k}\right), 
\quad n\in\mathbb N.
\end{align}
where the coefficients $a_n$ are defined in \eqref{seq_a}. 

In Table \ref{tab:analytical_values}, we put the selected values of $F^{(1)}_{\infty},\,F^{(2)}_{\infty},\,F^{(3)}_{\infty}$ computed by \eqref{main_ex_1}, \eqref{main_ex_2}, and \eqref{main_ex_3} accordingly. 

\begin{table}[H]
\centering
\captionsetup{justification=centering}
\begin{tabular}{c|ccccccc}
\hline
$n$ & 1 & 2 & 3 & 5 & 10 & 20 & 30 \\
\hline
$F_{\infty}^{(1)}(n)$ & 0.0696 & 0.1304 & 0.1860 & 0.2859 & 0.4850 & 0.7320 & 0.8606\\
\hline
$F_{\infty}^{(2)}(n)$ & 0.1801 & 0.2315 & 0.2799 & 0.3681 & 0.5442 & 0.7629 & 0.8766\\
\hline
$F_{\infty}^{(3)}(n)$ & 0.0440 & 0.0986 & 0.1538 & 0.2568 & 0.4639 & 0.7211 & 0.8549\\
\hline
\end{tabular}
\caption{Theoretical values of $F_{\infty}^{(1)}, F_{\infty}^{(2)}, F_{\infty}^{(3)}$. Numbers are rounded to four decimal places.}
\label{tab:analytical_values}
\end{table}

The agreement between the values presented in Tables \ref{tab:numerical_values} and \ref{tab:analytical_values} confirms the correctness of the formulas derived in Section~\ref{sec:results}. Moreover, Proposition \ref{T1} for $s\in\mathbb{C}$ such that $|s|<1$ and $G_{S_3}(s)\neq1$  yields the explicit expressions of the generating functions
\begin{align*}
\frac{G_{\mathcal M_1}(s)}{1-s}
&=
\frac{\left(\frac{1}{e-1}-\frac{15}{22}\right)(\alpha-s)}
{(1-\alpha)\left(11(e-1)e^{(s-1)/2}-s^2(20-9s)(e-s)\right)}
\,11s(e-s),\\
\frac{G_{\mathcal M_2}(s)}{1-s}
&=
\frac{\left(\frac{1}{e-1}-\frac{15}{22}\right)(\alpha-s)}
{(1-\alpha)\left(11(e-1)e^{(s-1)/2}-s^2(20-9s)(e-s)\right)}
\,(20-9s)(e-s),\\
\frac{G_{\mathcal M_3}(s)}{1-s}
&=
\frac{\left(\frac{1}{e-1}-\frac{15}{22}\right)(\alpha-s)}
{(1-\alpha)\left(11(e-1)e^{(s-1)/2}-s^2(20-9s)(e-s)\right)}
\,11s(e-1).
\end{align*}
Derivatives of $G_{\mathcal M_1}(s)$, $G_{\mathcal M_2}(s)$ and $G_{\mathcal M_3}(s)$, evaluated as $s\to1^-$, respectively, give
\begin{align*}
\mathbb{E}\mathcal{M}_1
&=
\frac{
e(506-298\alpha)
+37(-11+7\alpha)
+e^2(-187+127\alpha)
}
{4(e-1)(15e-37)(1-\alpha)}
\approx15.6653,
\\[2mm]
\mathbb{E}\mathcal{M}_2
&=
\frac{
-37(41+3\alpha)
+2e(703+441\alpha)
+e^2(-857+197\alpha)
}
{44(e-1)(15e-37)(1-\alpha)}
\approx13.8471,
\\[2mm]
\mathbb{E}\mathcal{M}_3
&=
\frac{
37(7-3\alpha)
+e(-187+127\alpha)
}
{4(15e-37)(1-\alpha)}
\approx16.2473,
\end{align*}
where $\alpha\approx-0.364796$ is a root of $G_{S_3}(s)=1$.

\begin{ex}\label{ex:classical}(Homogenous biased Rademacher random walk)
Let $N=1$ and suppose $X_1,\,X_2,\,X_3,\ldots$ are independent copies of $X$ whose distribution is
\begin{align*}
\mathbb P(X=-1)=\frac{1}{2}+\varepsilon=1-\mathbb{P}(X=1),
\quad
0<\varepsilon<\frac{1}{2}.
\end{align*}
We determine the distribution function $F^{(1)}_\infty$, generating function and moments of the random variable $\mathcal{M}_1$.
\end{ex}
According to Proposition \ref{T3} and Lemma \ref{L1},
\begin{align*}
F^{(1)}_\infty(-1)=-\mathbb{E}X=2\varepsilon,\quad 0<\varepsilon<\frac{1}{2}.
\end{align*}
By proposition \ref{T1},
\begin{align*}
\frac{G_{\mathcal{M}_1}(s)}{1-s}
=
\frac{4\varepsilon}
{(1-s)\left((1+2\varepsilon)-(1-2\varepsilon)s\right)}
=\frac{1}{1-s}
-
\frac{\rho}
{1-\rho s},\quad \rho=\frac{1-2\varepsilon}{1+2\varepsilon},\quad |s|<1,
\end{align*}
which implies
\begin{align}\label{clasical_ex}
\frac{G_{\mathcal{M}_1}(s)}{1-s}=\sum_{n=0}^{\infty}(1-\rho^{n+1})s^n \quad \Rightarrow \quad
F^{(1)}_\infty(n)
=
1-
\left(
\frac{1-2\varepsilon}
{1+2\varepsilon}
\right)^{n+1},
\quad
n\in\mathbb N_0.
\end{align}

Thus, statements from Section \ref{sec:results} immediately recover the classical formula \eqref{clasical_ex} for the distribution of the maximum of the biased Rademacher random walk; see, for example, \cite{Ross2014}, \cite{KarlinTaylor1975}. Moreover, by \eqref{clasical_ex},
\begin{align*}
G_{\mathcal{M}_1}\left(e^s\right)=\frac{1-\rho}{1-\rho e^s}
=(1-\rho)\sum_{k=0}^{\infty}\rho^k e^{ks}
=(1-\rho)\sum_{n=0}^{\infty}\left(\sum_{k=0}^{\infty}k^n\rho^k\right)\frac{s^n}{n!},\quad |s|<-\log\rho.
\end{align*}
Therefore,
\begin{align*}
\mathbb{E}\mathcal{M}_1^n=(1-\rho){\text{Li}}_{-n}(\rho),\qquad {\text{Li}}_{-n}(\rho)=\sum_{k=1}^{\infty}k^n\rho^k,\qquad n\in\mathbb{N}. 
\end{align*}
In particular,
\begin{align*}
\mathbb{E}\mathcal{M}_1=\frac{1-2\varepsilon}{4\varepsilon},\qquad 
\mathbb{E}\mathcal{M}_1^2=\frac{(1-2\varepsilon)}
{8\varepsilon^2}.
\end{align*}
As $\mathcal{M}_1=\max\{0,\,X_1,\,X_1+X_2,\,\ldots\}$, it follows that
\begin{align*}
\mathbb{E}\max\{X_1,\,X_1+X_2,\,\ldots\}
=
\mathbb{E}\mathcal{M}_1-F_{\infty}^{(1)}(-1)
=
\frac{(1-4\varepsilon)(1+2\varepsilon)}{4\varepsilon},
\end{align*}
where the higher moments of $\max\{X_1,\,X_1+X_2,\,\ldots\}$ can be computed analogously. In particular, when $\varepsilon=1/4$, although the random walk with increments valued in $\{-1,\,1\}$ has drift $\mathbb{E}X=-1/2$, the expected value of its all-time maximum is $0$, while its variance equals $3/2$.

\begin{ex}\label{ex_periodic_rademacher}(Bi seasonal biased Rademacher random walk)
Let $N=2$ and suppose $X_1,\,X_2,\,X_3,\ldots$ are independent random variables such that
$X_i\dist X_{i+2}$ for all $i\in\mathbb{N}$, and
\begin{align*}
\mathbb P(X_1=-1)=\frac12+\varepsilon=1-\mathbb{P}(X_1=1),
\qquad
\mathbb P(X_2=-1)=\frac12+\delta=1-\mathbb{P}(X_2=1),
\end{align*}
where $\varepsilon$ and $\delta$ are such that
\begin{align*}
-\frac12<\varepsilon,\,\delta<\frac12, \qquad \varepsilon+\delta>0.
\end{align*}
We determine the distribution functions $F_\infty^{(1)}$ and $F_\infty^{(2)}$, provide their generating functions, and compute the explicit moments of the random variables $\mathcal{M}_1$ and $\mathcal{M}_2$.
\end{ex}
In this example $\mathbb{E}S_2=-2(\varepsilon+\delta)<0$ and 
\begin{align*}
G_{S_2}(s)
=
\frac{(1+2\varepsilon)(1+2\delta)}{4s^2}
+
\frac{1-4\varepsilon\delta}{2}
+
\frac{(1-2\varepsilon)(1-2\delta)}{4}s^2=1\quad \Rightarrow
\end{align*}
\begin{align*}
\frac{(1-2\varepsilon)(1-2\delta)}{4s^2}
\left(s^2-1\right)
\left(
s^2-
\frac{(1+2\varepsilon)(1+2\delta)}
{(1-2\varepsilon)(1-2\delta)}
\right)=0.
\end{align*}
has a root $\alpha:=-1$ on the boundary of the unit disk. According to Lemma \ref{L1}, 
\begin{align*}
\begin{pmatrix}
1/2+\delta &
-\left(1/2+\varepsilon\right)
\\[2mm]
1/2+\delta &
1/2+\varepsilon
\end{pmatrix}
\begin{pmatrix}
\pi_0^{(1)}\\
\pi_0^{(2)}
\end{pmatrix}
=
\begin{pmatrix}
0\\
2(\varepsilon+\delta)
\end{pmatrix}\quad \Rightarrow \quad
\pi_0^{(1)}
=
\frac{2(\varepsilon+\delta)}
{(1+2\delta)},\quad
\pi_0^{(2)}
=
\frac{2(\varepsilon+\delta)}
{(1+2\varepsilon)}.
\end{align*}

Proposition \ref{T3} gives 
\begin{align*}
F_\infty^{(1)}(-1)
=
F_\infty^{(2)}(-1)
=
\varepsilon+\delta,
\end{align*}
while Proposition \ref{T1}, for the well-defined $s$, yields

\begin{align}\label{rational_f_1}
\frac{G_{\mathcal M_1}(s)}{1-s}
&=
\frac{
2(\varepsilon+\delta)
\left(
(1+2\varepsilon)
+2s
+(1-2\varepsilon)s^2
\right)
}
{(1+2\varepsilon)(1+2\delta)(1-s^2)(1-Rs^2)
} \quad \Rightarrow \quad
G_{\mathcal{M}_1}(s)=\frac{1-R}{2}\frac{1+2\varepsilon+(1-2\varepsilon)s}{1-Rs^2},
\\
\label{rational_f_2}
\frac{G_{\mathcal M_2}(s)}{1-s}
&=
\frac{
2(\varepsilon+\delta)
\left(
(1+2\delta)
+2s
+(1-2\delta)s^2
\right)
}
{(1+2\varepsilon)(1+2\delta)(1-s^2)(1-Rs^2)
}\quad \Rightarrow \quad
G_{\mathcal{M}_2}(s)=\frac{1-R}{2}\frac{1+2\delta+(1-2\delta)s}{1-Rs^2}.
\end{align}
where
\begin{align*}
R:=\frac{(1-2\varepsilon)(1-2\delta)}{(1+2\varepsilon)(1+2\delta)}\in(0,\,1).
\end{align*}
Since
\begin{align*}
\frac{1}{(1-s^2)(1-Rs^2)}=\frac{1}{1-R}\sum_{n=0}^{\infty}(1-R^{n+1})s^{2n},\quad |s|<1,
\end{align*}
and
\begin{align*}
1-R=\frac{4(\varepsilon+\delta)}{(1+2\varepsilon)(1+2\delta)}\quad \Rightarrow \quad
\frac{2(\varepsilon+\delta)}{(1+2\varepsilon)(1+2\delta)(1-R)}=\frac{1}{2},
\end{align*}
for $|s|<1$, from \eqref{rational_f_1} and \eqref{rational_f_2}, we obtain
\begin{align*}
\frac{G_{\mathcal M_1}(s)}{1-s}&=\frac{1+2\varepsilon
+2s
+(1-2\varepsilon)s^2}{2}\sum_{n=0}^{\infty}(1-R^{n+1})s^{2n},\\
\frac{G_{\mathcal M_2}(s)}{1-s}&=
\frac{1+2\delta
+2s
+(1-2\delta)s^2}{2}\sum_{n=0}^{\infty}(1-R^{n+1})s^{2n}
.
\end{align*}
Therefore, for all $n\in\mathbb{N}_0$,
\begin{align*}
&F_\infty^{(1)}(2n)
=
\frac{(1+2\varepsilon)\left(1-R^{n+1}\right)
+
(1-2\varepsilon)\left(1-R^{n}\right)}{2}
,
\qquad
F_\infty^{(1)}(2n+1)
=
1-R^{n+1},\\
&F_\infty^{(2)}(2n)
=
\frac{(1+2\delta)\left(1-R^{n+1}\right)
+
(1-2\delta)\left(1-R^{n}\right)}{2}
,
\qquad
F_\infty^{(2)}(2n+1)
=
F_\infty^{(1)}(2n+1).
\end{align*}
Consequently, for all $k\in\mathbb{N}_0$,
\begin{align*}
\pi^{(1)}_{2k}&=F_{\infty}^{(1)}(2k)-F_{\infty}^{(1)}(2k-1)=\frac{1+2\varepsilon}{2}(1-R)R^k,\\
\pi^{(1)}_{2k+1}&=F_{\infty}^{(1)}(2k+1)-F_{\infty}^{(1)}(2k)=\frac{1-2\varepsilon}{2}(1-R)R^k,\\
\pi^{(2)}_{2k}&=\frac{1+2\delta}{2}(1-R)R^k,\\
\pi^{(2)}_{2k+1}&=\frac{1-2\delta}{2}(1-R)R^k.
\end{align*}
and, for all $n\in\mathbb{N}$, given that $0<R<1$,
\begin{align*}
\mathbb{E}\mathcal{M}_1^n
&=
\frac{1-R}{2}
\sum_{k=0}^{\infty}
\left(
(1+2\varepsilon)(2k)^n
+
(1-2\varepsilon)(2k+1)^n
\right)
R^k,\\
\mathbb{E}\mathcal{M}_2^n
&=
\frac{1-R}{2}
\sum_{k=0}^{\infty}
\left(
(1+2\delta)(2k)^n
+
(1-2\delta)(2k+1)^n
\right)
R^k.
\end{align*}
From this 
\begin{align*}
&\mathbb E\mathcal M_1
=
\frac{(1-2\varepsilon)(1+\varepsilon-\delta)}
{2(\varepsilon+\delta)},
\qquad
\mathbb{E}\mathcal{M}_1^2
=
\frac{(1-2\varepsilon)
\left(
1+(\varepsilon-\delta)
-(\varepsilon-\delta)^2
+4\varepsilon\delta(\varepsilon-\delta)
\right)}
{2(\varepsilon+\delta)^2}\\
&\mathbb E\mathcal M_2
=
\frac{(1-2\delta)(1+\delta-\varepsilon)}
{2(\varepsilon+\delta)},
\qquad
\mathbb E\mathcal M_2^2
=
\frac{(1-2\delta)
\left(
1+(\delta-\varepsilon)
-(\delta-\varepsilon)^2
+4\delta\varepsilon(\delta-\varepsilon)
\right)}
{2(\varepsilon+\delta)^2}.
\end{align*}

Setting $\delta=\varepsilon$ in Example~\ref{ex_periodic_rademacher} recovers the generating function and distribution of Example~\ref{ex:classical}. Moreover, when $\delta=\varepsilon$ and $n\in\mathbb{N}$, it can be shown that
\begin{align*}
\mathbb{E}\mathcal{M}_1^n=\mathbb{E}\mathcal{M}_2^n=(1-\rho)2^n
\left(
\mathrm{\Phi}\left(\rho^2,\,-n,\,0\right)
+
\rho\,
\mathrm{\Phi}\!\left(\rho^2,\,-n,\,\frac12\right)\right)=(1-\rho)\text{Li}_{-n}(\rho),\quad \rho=\frac{1-2\varepsilon}{1+2\varepsilon},
\end{align*}
where 
\begin{align*}
{\text{Li}}_{-n}(\rho)=\sum_{k=1}^{\infty}k^n\rho^k,\quad \mathrm{\Phi}(z,\,s,\,a)=\sum_{k=0}^{\infty}\frac{z^k}{(k+a)^s},
\qquad |z|<1.
\end{align*}
It can also be shown that
$\mathbb{E}\max\{X_1,\,X_1+X_2,\,\ldots\}
=
\mathbb{E}\max\{X_2,\,X_2+X_1,\,\ldots\}=0
$
iff $\varepsilon=\delta=1/4$.

\begin{ex}(Homogeneous double root example)
Let $N=1$ and suppose $X_1,X_2,\ldots$ are independent copies of discrete random variable $X$ whose probability mass is
$$
\mathbb{P}(X=-3)=\frac{5}{54},\qquad
\mathbb{P}(X=-2)=\frac{11}{27},\qquad
\mathbb{P}(X=1)=\frac{1}{2}.
$$
We determine the distribution function $F^{(1)}_{\infty}=F_{\infty}$ and moments of the random variable $\mathcal{M}$. For convenience, we omit the upper index $^{(1)}$ from other involved notations as well.
\end{ex}

In this example $m_1=-3$, $\mathbb{E}X=-16/27$, and 
\begin{align*}
G_X(s)=\frac{5}{54}s^{-3}+\frac{11}{27}s^{-2}+\frac{1}{2}s=1
\end{align*}
has a double root $\alpha=-1/3$ inside unit disk.

According to Lemma \ref{L1} and instructions beneath it (see {\sc Note 1}), under the existence of double roots, we set up the following system
\begin{align*}
\begin{pmatrix}
F_X(-3)+\alpha F_X(-2)+\alpha^2F_X(-1)
&
\alpha F_X(-3)+\alpha^2F_X(-2)
&
\alpha^2F_X(-3)
\\[2mm]
F_X(-2)+2\alpha F_X(-1)
&
F_X(-3)+2\alpha F_X(-2)
&
2\alpha F_X(-3)
\\[2mm]
F_X(-3)+F_X(-2)+F_X(-1)
&
F_X(-3)+F_X(-2)
&
F_X(-3)
\end{pmatrix}
\begin{pmatrix}
\pi_0\\
\pi_1\\
\pi_2
\end{pmatrix}
=
\begin{pmatrix}
0\\
0\\
-\mathbb{E}X
\end{pmatrix},
\end{align*}
which implies
\begin{align*}
\begin{pmatrix}
-1/54& 2/81 & 5/486
\\[2mm]
1/6 & -13/54 & -5/81
\\[2mm]
59/54 & 16/27 & 5/54
\end{pmatrix}
\begin{pmatrix}
\pi_0\\
\pi_1\\
\pi_2
\end{pmatrix}
=
\begin{pmatrix}
0\\
0\\
16/27
\end{pmatrix}
\quad
\Rightarrow
\quad
\pi_0=\frac{2}{5},\qquad
\pi_1=\frac{6}{25},\qquad
\pi_2=\frac{18}{125}.
\end{align*}

According to Proposition \ref{T3}, 
\begin{align*}
&F_\infty(-3)=\frac{1}{27},\quad
F_\infty(-2)=\frac{2}{9},\quad
F_\infty(-1)=\frac{1}{3}.
\end{align*}

Proposition \ref{T1} gives
\begin{align*}
\frac{G_{\mathcal{M}}(s)}{1-s}
=
\frac{2}{(1-s)(5-3s)}=\frac{1}{1-s}-\frac{3/5}{1-3/5\, s}\quad \Rightarrow \quad F_{\infty}(n)=1-\left(\frac{3}{5}\right)^{n+1},\,n\in\mathbb{N}_0.
\end{align*}
As in Example \ref{ex:classical},
\begin{align*}
\mathbb{E}\mathcal{M}^n=\frac{2}{5}\sum_{k=1}^{\infty}k^n\left(\frac{3}{5}\right)^k,\quad n\in\mathbb{N}.
\end{align*}
In particular
\begin{align*}
\mathbb{E}\mathcal{M}=\frac{3}{2},\qquad \mathbb{E}\mathcal{M}^2=6,
\qquad \mathbb{E}\mathcal{M}^3=\frac{141}{4},\qquad \ldots
\end{align*}
Interestingly, Example~\ref{ex:classical} with $\varepsilon=1/8$ yields the same probability generating function, and hence the same distribution of $\mathcal{M}$, as the present example.

\bibliography{x}
\end{document}